\documentclass[a4paper,oneside, 11 pt, reqno]{amsart}

\usepackage{enumerate}
\usepackage{amsmath,amsthm,amscd,amsfonts,amssymb,mathrsfs}
\usepackage{latexsym,amsmath}
\usepackage{mathrsfs}
\makeatletter
\@namedef{subjclassname@2010}{%
\textup{2010} Mathematics Subject Classification}
\makeatother

\newcommand{\ncom}{\newcommand}
\ncom{\beqn}{\begin{eqnarray*}}
\ncom{\eeqn}{\end{eqnarray*}}
\ncom{\beq}{\begin{eqnarray}}
\ncom{\eeq}{\end{eqnarray}}
\ncom{\cal}{\mathcal}
\ncom{\eop}{\hfill{{\rule{2.5mm}{2.5mm}}}}
\ncom{\eoe}{\hfill{{\rule{1.5mm}{1.5mm}}}}
\ncom{\eof}{\hfill{{\rule{1.5mm}{1.5mm}}}}
\ncom{\hone}{\mbox{\hspace{1em}}}
\ncom{\htwo}{\mbox{\hspace{2em}}}
\ncom{\hthree}{\mbox{\hspace{3em}}}
\ncom{\hfour}{\mbox{\hspace{4em}}}
\ncom{\hsev}{\mbox{\hspace{7em}}}
\ncom{\vone}{\vskip 2ex}
\ncom{\vtwo}{\vskip 4ex}
\ncom{\vonee}{\vskip 1.5ex}
\ncom{\vthree}{\vskip 6ex}
\ncom{\vfour}{\vspace*{8ex}}
\ncom{\norm}{\|\;\;\|}
\ncom{\integ}[4]{\int_{#1}^{#2}\,{#3}\,d{#4}}
\ncom{\inp}[2]{\langle{#1},\,{#2} \rangle}
\ncom{\Inp}[2]{\Langle{#1},\,{#2} \Langle}
\ncom{\vspan}[1]{{{\rm\,span}\#1 \}}}
\ncom{\dm}[1]{\displaystyle {#1}}

\newtheorem{theorem}{\bf Theorem}[section]
\newtheorem{proposition}[theorem]{\bf Proposition}
\newtheorem{corollary}[theorem]{\bf Corollary}

\newtheoremstyle
    {remarkstyle}
    {}
    {11pt}
    {}
    {}
    {\bfseries}
    {:}
    {     }
    {\thmname{#1} \thmnumber{#2} }

\theoremstyle{remarkstyle}

\newtheorem{remark}[theorem]{\bf Remark}
\newtheorem{definition}[theorem]{\bf Definition}
\newtheorem{example}[theorem]{\bf Example}

\begin{document}

\title[Hyperexpansive Weighted Translation Semigroups ]
{Hyperexpansive Weighted Translation Semigroups }

\author[G. M. Phatak]{Geetanjali M. Phatak}
\address{Department of Mathematics, S. P. College\\
Pune- 411030, India}
\email{gmphatak19@gmail.com}

\author[V. M. Sholapurkar]{V. M. Sholapurkar}
\address{ Department of Mathematics, S. P. College\\
Pune- 411030, India}
\email{vmshola@gmail.com}

\date{}

\begin{abstract}
The weighted shift operators turn out to be extremely useful in supplying interesting examples of operators on Hilbert spaces. With a view to study a continuous analogue of weighted shifts, M. Embry and A. Lambert initiated the study of a semigroup of operators $\{S_t\}$  indexed by a non-negative real number $t$ and termed it as weighted translation semigroup. The operators $S_t$ are defined on $L^2(\mathbb R_+)$ by using a weight function. In this paper, we continue the work carried out there and obtain characterizations of hyperexpansive weighted translation semigroups in terms of their symbols. We also discuss Cauchy dual of a hyperexpansive weighted translation semigroup. As an application of the techniques developed, we present new proofs of a couple of known results.  
\end{abstract}
\maketitle
\section{Introduction}
Let $H$ be a complex separable Hilbert space and $\cal B(H)$ the algebra of bounded linear operators on $H.$ Recall that for an orthonormal basis $\{e_n\}$ of $H$ and a bounded sequence $\{\alpha_n\}$ of scalars, a weighted shift operator $T:\{\alpha_n\}$, is defined by $T(e_n)=\alpha_n e_{n+1}$ and extended by linearity and continuity. The flexibility of choosing a weight sequence $\{\alpha_n\}$ allows one to construct several interesting examples of a variety of classes of operators. The class of weighted  shift operators has been systematically studied in \cite{Se}. With a view to develop a continuous analogue of weighted shifts, M. Embry and A. Lambert \cite{EL1} initiated the study of operators that can be defined by using a {\it weight function} (rather than a weight sequence). In fact, for a positive, measurable function $\varphi,$ they constructed a semigroup $\{S_t\}$ of bounded linear operators on $L^2(\mathbb R_+),$ parametrized by a non-negative real number $t$ and termed it as {\it weighted translation semigroup}. In this paper, we further explore the semigroup $\{S_t\},$ provide a variety of examples and present some important properties. 
In section 2, we set the notation and record some definitions required in the sequel. 
In section 3, we study some special types of weighted translation semigroups $\{S_t\}.$ In  \cite{EL1} and \cite{EL2}, M. Embry and A. Lambert characterized hyponormal and subnormal weighted translation semigroups. Here, capitalizing on the theory of associating a special type of non-negative function to a special type of bounded linear operator as developed in \cite{At} and \cite{SA}, we provide characterizations for completely hyperexpansive, 2-hyperexpansive, 2-isometric and alternatingly hyperexpansive weighted translation semigroups in terms of their symbols. In this paper, these four classes of operator semigroups are referred to as hyperexpansive weighted translation semigroups. In the process, we also obtain characterization of a contractive subnormal weighted translation semigroup in terms of a completely monotone function. 
Section 4 is devoted to the notion of Cauchy dual of a hyperexpansive weighted translation semigroup. We present a comparative analysis of weighted shift operators and weighted translation semigroups, especially in the light of Cauchy duals of corresponding classes and present new proofs of some known results.    
 
\section{Preliminaries}
Let $\mathbb R_+$ be the set of non-negative real numbers and let $L^2({\mathbb R_+})$ denote the Hilbert space of complex valued square integrable Lebesgue measurable functions on $\mathbb R_+.$ Let ${\cal B}(L^2)$ denote the algebra of bounded linear operators on $L^2({\mathbb R_+}).$ 
\begin{definition}
For a measurable, positive function $\varphi$ defined on $\mathbb R_+$ and $t\in \mathbb R_+,$ define the function $\varphi_t : \mathbb R_+ \rightarrow \mathbb R_+$ by 
\begin{equation*}
\varphi_t(x) =
\begin{cases}
\displaystyle \sqrt {\frac{\varphi(x)}{\varphi(x-t)}} & \text{if~ $x\geq t$},\\
0 & \text{if~ $x<t$}.
\end{cases}
\end{equation*}
Suppose that $\varphi_t$ is essentially bounded for every $t \in \mathbb R_+$.
For each fixed $t\in \mathbb R_+,$ we define $S_t$ on $L^2({\mathbb R_+})$ by  
\begin{equation*}
S_tf(x) =
\begin{cases}
\varphi_t(x)f(x-t) & \text{if~ $x\geq t$},\\
0 & \text{if~ $x<t$}.
\end{cases}
\end{equation*}
\end{definition}
\begin{remark}
Substituting $\varphi_t$ in the above definition, we get
\begin{equation*}
S_tf(x) =
\begin{cases}
\displaystyle \sqrt {\frac{\varphi(x)}{\varphi(x-t)}}f(x-t) & \text{if~ $x\geq t$},\\
0 & \text{if~ $x<t$}.
\end{cases}
\end{equation*}
It is easy to see that for every $t\in \mathbb R_+,~S_t$ is a bounded linear operator on $L^2({\mathbb R_+})$ with $\|S_t\|=\|\varphi_t\|_\infty,$  where $\|\varphi_t\|_\infty$ stands for the essential supremum of $\varphi_t$ given by \beqn \|\varphi_t\|_\infty= \inf  \{M \in \mathbb R: \varphi_t(x)\leq M ~\rm{almost~everywhere}\}. \eeqn
The family $\{S_t:t\in \mathbb R_+\}$ in ${\cal B}(L^2)$ is a semigroup with $S_0=I,$ the identity operator and for all $t,s~\in \mathbb R_+$, $S_t\circ S_s=S_{t+s}.$ 
\end{remark}

We say that $\varphi_t$ is a {\it weight function corresponding to the operator $S_t$}. Further, the semigroup $\{S_t:t\in \mathbb R_+\}$ is referred to as the {\it weighted translation semigroup with symbol $\varphi$}.
Throughout this article, we assume that the symbol $\varphi$ is a continuous function on $\mathbb R_+.$ 
We now turn our attention towards extending the association of special classes of functions and special types of operator semigroups $\{S_t\}.$ We now record the definitions of the special types of functions under consideration for the sake of completeness. Recall that a function $f$ is said to be of {\it class $C^n$} if it is differentiable $n$ times and the $n^{\rm{th}}$ derivative is continuous. A function $f$ is said to be of {\it class $C^\infty$} if it is differentiable infinitely many times.

A $C^\infty$ function $f:\mathbb R_+ \rightarrow \mathbb R$ is called  
\begin{enumerate} 
\item 
{\it completely monotone} if $$(-1)^kf^{(k)}(x)\geq 0, ~{\rm for~ all}~ k\geq 0,$$ where $f^{(k)}$ denotes the k$^{\rm {th}}$ derivative of $f.$
\item 
{\it completely alternating} if $$(-1)^{k-1} f^{(k)}(x)\geq 0,~~\mbox{for all}~ k\geq 1.$$
\item 
{\it absolutely monotone} if $$f^{(k)}(x)\geq 0, ~{\rm for~ all}~ k\geq 0.$$ 
\end{enumerate}

These classes of functions have been extensively studied in the literature \cite{BCR},\cite{SSV},\cite{W-1} and they get naturally associated with some special classes of operators. We find it convenient to record the definitions of the classes of operators under consideration for ready reference.  
Let $T$ be a bounded linear operator on a Hilbert space $H$ and $n$ be a positive integer. 
Let $B_n(T)$ denote the operator
\begin{equation} \label{B-n-T}
B_n(T)= \sum_{k=0}^n(-1)^k{n\choose k} {T^*}^kT^k,~ n \geq 1.
\end{equation}
An operator $T$ is said to be 
\begin{enumerate} 
\item  {\it subnormal} if there exist a Hilbert space $K$ containing $H$ and a normal operator $N\in\cal B(K)$ such that $N{H}\subseteq H$ and $N|_{H}=T.$
\item {\it completely hyperexpansive} if $B_n(T) \leq 0,~ \rm{for~all~integers~} n\geq 1.$
\item {\it $m$-hyperexpansion} if $B_n(T) \leq 0,~ \rm{for~all~integers~} n,~ 1\leq n\leq m.$
\item {\it alternatingly hyperexpansive} if $$\sum_{k=0}^n(-1)^{n-k}{n\choose k} {T^*}^kT^k\geq 0, ~\rm{for~all~integers~} n\geq 1.$$  
\item {\it $m$-isometry} if $B_m(T)= 0.$ 
\item {\it hyponormal} if $T^*T-TT^*\geq 0.$ 
\item {\it contraction (expansion)} if $I-T^*T\geq 0 ~~(I-T^*T\leq 0).$
\end{enumerate} 
For a detailed account on these classes of operators, the reader is referred to \cite{AS},\cite{At},\cite{Co},\cite{Jb1},\cite{Jb2},\cite{SA}. 

We now quote a characterization of subnormal contractions given by J. Agler \cite{Ag} which is used in the sequel. 
\begin{theorem} \label{p1} An operator $T$ on a Hilbert space $\cal H$ is a subnormal contraction if and only if  $B_n(T) \geq 0,~ \rm{for~all~integers~} n\geq 1.$ 
\end{theorem} 

Among the classes of operators as defined above, hyponormal and subnormal weighted translation semigroups have been studied by M. Embry and \\A. Lambert. The characterizations of these semigroups in terms of their symbols as given there, are recorded here (\cite[Lemma 3.3]{EL1},  \cite[Theorem 2.2]{EL2}). 
\begin{theorem} \label{p2} The semigroup $\{S_t\}$ with a continuous, positive symbol $\varphi$ is hyponormal if and only if $\log \varphi$ is convex.\end{theorem}
The assumption of continuity of the function $\varphi$ in the above theorem is superfluous. It turns out that, if the semigroup $\{S_t\}$ is hyponormal, then the function $\log \varphi$ is mid-point convex. Now by a result of Sierpinski \cite{Si}, a measurable and mid-point convex function turns out to be continuous.     
\begin{theorem} \label{p3} The semigroup $\{S_t\}$ with a continuous symbol $\varphi$ is subnormal if and only if $$ \varphi (x)=\int_0^as^xd\rho (s),$$ where $\displaystyle a=\lim_{t\rightarrow \infty}||S_t^*S_t||^{\frac{1}{t}}$ and $\rho$ is a probability measure on $[0,a].$   \end{theorem}

\section{Characterizations of hyperexpansive weighted translation semigroups}
In this section, we continue the theme of characterizing some classes of weighted translation semigroups in terms of their symbols. We say that a semigroup $\{S_t\}$ is {\it 2-hyperexpansive}, {\it $m$-isometry}, {\it completely hyperexpansive}, {\it subnormal}, {\it hyponormal}, {\it alternatingly hyperexpansive} if each operator $S_t$ respectively belongs to that class. 
Consider a weighted translation semigroup $\{S_t\}$ with symbol $\varphi.$
Observe that the adjoint of $S_t$ is given by 
$$S_t^*f(x) = \sqrt {\frac{\varphi(x+t)}{\varphi(x)}}f(x+t) \quad \mbox{for all}~ x\geq 0.$$
Further it is easy to see that,
 $$ (S_t^*S_t)f(x)=\frac{\varphi(x+t)}{\varphi(x)}f(x) \quad \mbox{for all}~ x\geq 0.$$

\begin{theorem} \label{p4} 
Let $\{S_t\}$ be a weighted translation semigroup with symbol $\varphi.$ Let $n$ be a fixed positive integer and $\varphi \in C^n.$ Then the following statements are equivalent:
\begin{enumerate}
\item $B_n(S_t) \geq 0 ~\rm{for~ all}~ t\in \mathbb R_+.$
   
\item $\displaystyle \sum_{k=0}^n(-1)^k{n\choose k} \varphi(x+kt)\geq 0 ~{\rm for~all~} x,t\in \mathbb R_+.$

\item $(-1)^n\varphi^{(n)}(x) \geq 0 ~{\rm for~all~} x\in \mathbb R_+.$

\end{enumerate} \end{theorem}
\begin{proof} We first prove $(1)\Leftrightarrow (2).$ Note that for each $t\in \mathbb R_+,~ B_n(S_t)\geq 0 $
if and only if 
$\langle B_n(S_t)f,f \rangle \geq 0 ~{\rm for~all~} f\in L^2({\mathbb R_+}),~{\rm for~all~} t\in \mathbb R_+.$
This is true if and only if $$\int_0^\infty \left(\sum_{k=0}^n(-1)^k{n\choose k} \varphi(x+kt)\right)|f(x)|^2 dx\geq 0 ~{\rm for~all~} f\in L^2({\mathbb R_+}),~{\rm for~all~} x,t\in \mathbb R_+.$$
Now by virtue of continuity of $\varphi,$ this statement is equivalent to
$$\sum_{k=0}^n(-1)^k{n\choose k} \varphi(x+kt)\geq 0 ~{\rm for~all~} x,t\in \mathbb R_+.$$

We now prove $(2)\Rightarrow (3).$
Since $\varphi \in C^n,$ we have 
$$(-1)^n\varphi ^{(n)}(x)=\lim_{h\rightarrow 0}\frac{\displaystyle \sum_{k=0}^n(-1)^k{n\choose k} \varphi(x+kh)}{h^n}.$$
Hence $$\sum_{k=0}^n(-1)^k{n\choose k} \varphi(x+kt)\geq 0 ~{\rm for~all~} x,t\in \mathbb R_+,$$
implies that
$(-1)^n\varphi^{(n)}(x) \geq 0~{\rm for~all~} x\in \mathbb R_+.$
The proof of $(3)\Rightarrow (2)$ follows from the fact that the repeated application of Mean Value Theorem gives $$\sum_{k=0}^n(-1)^k{n\choose k} \varphi(x+kt)=(-1)^nt^n\varphi^{(n)}(x^\prime),$$ where $x^\prime \in (x,x+nt)$ (cf. \cite[Theorem 4.8]{SSV}). 
\end{proof}

The proof of the following theorem is similar to the above.
\begin{theorem} \label{p30} 
Let $\{S_t\}$ be a weighted translation semigroup with symbol $\varphi.$ Let $n$ be a fixed positive integer and $\varphi \in C^n.$ Then the following statements are equivalent:
\begin{enumerate}
\item $B_n(S_t) \leq 0~\rm{for~ all}~ t\in \mathbb R_+.$
   
\item $\displaystyle \sum_{k=0}^n(-1)^k{n\choose k} \varphi(x+kt)\leq 0 ~{\rm for~all~} x,t\in \mathbb R_+.$

\item $(-1)^n\varphi^{(n)}(x) \leq 0~{\rm for~all~} x\in \mathbb R_+.$

\end{enumerate} \end{theorem}

The following corollary now follows at once from Theorem \ref{p4} and \ref{p30}.
\begin{corollary} \label{p6} Let $\{S_t\}$ be a weighted translation semigroup with symbol $\varphi$ and $\varphi \in C^\infty.$ Then
\begin{enumerate}
\item The semigroup $\{S_t\}$ is 2-hyperexpansive if and only if $\varphi$ is a concave function.
\item The semigroup $\{S_t\}$ is $m$-isometry if and only if $\varphi$ is a polynomial of degree $m-1.$
\item The semigroup $\{S_t\}$ is completely hyperexpansive if and only if $\varphi$ is a completely alternating function.
\item The semigroup $\{S_t\}$ is subnormal contraction if and only if $\varphi$ is a completely monotone function.
\item The semigroup $\{S_t\}$ is alternatingly hyperexpansive if and only if $\varphi$ is an absolutely monotone function.
\end{enumerate} \end{corollary}
\begin{remark}
Note that if the semigroup $\{S_t\}$ is 2-hyperexpansive, then the symbol $\varphi$ is mid-point concave. As in case of convexity, it turns out that a measurable, mid-point concave function is continuous and hence concave. Thus the statement (1) in Corollary \ref{p6} can be proved only under the assumption of measurability of $\varphi.$ Also the statement (2) in Corollary \ref{p6} can be proved under weaker condition that $\varphi \in C^m.$   
\end{remark}
 
We now give the characterizations of contractive subnormal and completely hyperexpansive semigroups $\{S_t\},$ in terms of integral representations of their symbol $\varphi$ by appealing to the integral representations of completely monotone and completely alternating maps as given in \cite[Theorem 1.4, Theorem 3.2]{SSV}. 
\begin{proposition}  \label{p35} 
Let $\{S_t\}$ be a weighted translation semigroup with symbol $\varphi$ and $\varphi \in C^\infty.$ Then 
\begin{enumerate}
\item The semigroup $\{S_t\}$ is a subnormal contraction if and only if $$\varphi(x)=\int_0^\infty e^{-xa}d\mu (a),$$ where $\mu$ is a finite non-negative Borel measure on $[0,\infty).$ 
\item The semigroup $\{S_t\}$ is completely hyperexpansive if and only if $$\varphi(x)=\varphi (0)+cx+\int_0^\infty (1-e^{-ax})d\mu (a),$$ where $c$ is a non-negative real number and $\mu$ is a measure on $[0,\infty)$ satisfying $\displaystyle \int_0^\infty (1\wedge a)d\mu (a) < \infty.$ 
\end{enumerate}
\end{proposition}

\begin{remark} 
We now observe that if $\varphi$ is either completely monotone or completely alternating function, then the weight function $\varphi_t$ is essentially bounded. 
\begin{enumerate}
\item If the function $\varphi$ is completely monotone, then $\varphi ^\prime \leq 0.$ Therefore $\varphi$ is a decreasing function. Hence $\varphi_t \leq 1$ for all $t\in \mathbb R_+.$
\item If the function $\varphi$ is completely alternating, then $\varphi ^\prime \geq 0$ and therefore $\varphi$ is an increasing function. Hence $\varphi_t(x) \geq 1$ for $x\geq t.$ We now prove that for every fixed $t\in \mathbb R_+,\displaystyle \lim_{x\rightarrow \infty} \varphi_t (x)=1.$
It is sufficient to prove that $$\lim_{x\rightarrow \infty} \varphi_t^2 (x)=\lim_{x\rightarrow \infty} \frac{\varphi (x)}{\varphi (x-t)}=1.$$
Since $\varphi$ is a completely alternating function, by Proposition \ref{p35}(2), we have $$ \varphi (x)=\varphi (0)+cx+\int_0^\infty (1-e^{-ax})d\mu (a),$$ where $c$ is a non-negative real number and $\mu$ is a measure on $[0,\infty)$ satisfying $\displaystyle \int_0^\infty (1\wedge a)d\mu (a) < \infty.$
Now, 
\beqn 
\frac{\varphi (x)}{\varphi (x-t)}-1 
& = & \frac{\varphi (x)-\varphi (x-t)}{\varphi (x-t)} \\
& = & \frac{ct+\int_0^\infty [e^{-a(x-t)}-e^{-ax}]d\mu (a)}{\varphi (0)+c(x-t)+\int_0^\infty (1-e^{-a(x-t)})d\mu (a)}. \eeqn
Note that $F(x)=e^{-a(x-t)}-e^{-ax}$ is a decreasing function and $F(x)$ tends to zero as $x\rightarrow \infty.$
By Lebesgue monotone convergence theorem, the integral $\int_0^\infty [e^{-a(x-t)}-e^{-ax}]d\mu (a)$ converges to zero. Hence $\displaystyle \lim_{x\rightarrow \infty} \varphi_t (x)=1.$
\end{enumerate}
\end{remark}

As noted in Corollary \ref{p6}(1), the symbol for a 2-hyperexpansive weighted translation semigroup is concave.
 Now in view of \cite[Lemma 3.16]{CS} a non-negative concave $C^2$ function is increasing. Thus if the symbol $\varphi $ for a 2-hyperexpansive weighted translation semigroup is $C^2,$ then it is an increasing function. Recall that if a weighted shift operator is 2-hyperexpansive, then the corresponding weight sequence is decreasing \cite[Proposition 4]{At}. However, the following example indicates that the weight function $\varphi_t$ associated to 2-hyperexpansive operator $S_t$ need not be monotonic. 
\begin{example}
Let $\varphi (x)=\sqrt{x+1}.$ In this case, the function $\varphi$ is concave and therefore by Corollary \ref{p6} (1), the corresponding semigroup $\{S_t\}$ is 2-hyperexpansive. We now show that $\varphi_1,$ the weight function corresponding to $S_1,$ is neither increasing nor decreasing. 
\begin{equation*}
\varphi_1(x) =
\begin{cases}
\displaystyle \frac{\sqrt[4]{x+1}}{\sqrt[4]{x}} & \text{if~ $x\geq 1$}\\
0 & \text{if~ $x<1$}.
\end{cases}
\end{equation*}

For $x=2, y=3,$ we have $\varphi_1(x)=\sqrt[4]{\frac{3}{2}} > \varphi_1(y)=\sqrt[4]{\frac{4}{3}}$ and if $x=\frac{1}{2}, y=2,$ then $\varphi_1(x)=0 < \varphi_1(y)=\sqrt[4]{\frac{3}{2}}.$ Thus the weight function $\varphi_1$ is neither increasing nor decreasing.
\eop
\end{example}

We now illustrate the theme developed in Corollary \ref{p6} to construct some special types of weighted translation semigroups by choosing some special types of symbols. 
\begin{example} Let $\varphi$ be the symbol for the semigroup $\{S_t\}.$ In the following examples, the weight function $\varphi_t$ is essentially bounded for every $t \in \mathbb R_+$.
\begin{enumerate}
\item
Let $\varphi(x)=\sqrt{x+1},$  
\begin{equation*}
\varphi(x) =
\begin{cases}
x+1 & \text{if~ $0\leq x\leq 1$},\\
2 & \text{if~ $x>1$}.
\end{cases}
\end{equation*}
These functions are concave. Thus by Corollary \ref{p6} (1), the semigroups $\{S_t\}$ corresponding to these symbols $\varphi$ are 2-hyperexpansive.
\item
If $\varphi$ is a polynomial of degree $m-1,$ then by Corollary \ref{p6} (2) the corresponding semigroup $\{S_t\}$ is an $m$-isometry. 
\item
Let $\varphi(x)= \log (x+2), \varphi(x)=2-e^{-x}, \varphi(x)=x+1.$
These functions are completely alternating. Thus by Corollary \ref{p6} (3), the semigroups $\{S_t\}$ corresponding to these symbols $\varphi$ are completely hyperexpansive.
\item
Let $\varphi(x)=\frac{1}{x+1},~ \varphi(x)=e^{-x}.$ These functions are completely monotone. Thus by Corollary \ref{p6} (4), the semigroups $\{S_t\}$ corresponding to these symbols $\varphi$ are  subnormal contractions.
\item
Let $\displaystyle \varphi(x)=\sum_{k=0}^n a_kx^k,~a_0,a_1,\cdots ,a_n>0.$
This function is absolutely monotone. Thus by Corollary \ref{p6} (5), the semigroup $\{S_t\}$ corresponding to the symbol $\varphi$ is alternatingly hyperexpansive.
\item
Let $\displaystyle \varphi (x)=\frac{x+\lambda}{x+1},~\lambda > 0.$ 
Then $\displaystyle \varphi ^{(n)}(x)=\frac{(1-\lambda)(-1)^{n-1}n!}{(x+1)^{n+1}}.$\\
If $0< \lambda <1,$ then $\varphi ^{(n)}>0$ for n odd and $\varphi ^{(n)}<0$ for n even.
In this case, $\varphi$ is a completely alternating function and the corresponding semigroup $\{S_t\}$ is completely hyperexpansive.
If $\lambda >1,$ then $\varphi ^{(n)}<0$ for n odd and $\varphi ^{(n)}>0$ for n even.
In this case, $\varphi$ is a completely monotone function and the corresponding semigroup $\{S_t\}$ is a subnormal contraction. 
\item
Let $\varphi(x)=e^x.$ Since $\log \varphi$ is convex, the corresponding semigroup $\{S_t\}$ is hyponormal. By Corollary \ref{p6} (4), this semigroup is not subnormal.
\eop 
\end{enumerate} \end{example}

In \cite[Theorem 1]{Fi}, it was proved that if a function $\varphi$ is completely monotone, then $\log \varphi$ is convex. We now present a different proof of the this fact using Corollary \ref{p6}.
\begin{proposition} \label{p12} If a function $\varphi$ is completely monotone, then $\log \varphi$ is convex.\end{proposition}
\begin{proof} Suppose a positive $C^\infty$ function $\varphi$ is completely monotone. Then the weight function $\varphi_t$ is essentially bounded. Therefore by Corollary \ref{p6} (4), the semigroup $\{S_t\}$ with symbol $\varphi$ is subnormal. This implies that the semigroup $\{S_t\}$ is hyponormal \cite[Proposition 4.2]{Co}. Hence by Theorem \ref{p2}, the function $\log \varphi$ is convex. \end{proof}

\section{Cauchy dual of hyperexpansive weighted translation semigroups $\{S_t\}$}
A notion of the {\it  Cauchy dual} of a left invertible operator was introduced by S. Shimorin \cite{Sh}. Recall that for a left invertible operator $T,$ the {\it Cauchy dual} $T^\prime$ of $T$ is defined as $T^\prime =T(T^*T)^{-1}.$ Note that an operator $T$ is left invertible if and only if $T$ is injective and range of $T$ is closed. 

In this section, we shall discuss the Cauchy dual of a hyperexpansive weighted translation semigroup. In particular, the Cauchy duals of completely hyperexpansive, 2-hyperexpansive and 2-isometric weighted translation semigroups have been discussed. Though the Cuachy duals of each of these classes of operators have been dealt with in the literature \cite{AC},\cite{At},\cite{Ch},\cite{Sh1}, the function theoretic considerations allow one to give considerably simpler proofs in the context of a weighted translation semigroup.       
Let $\{S_t\}$ be a weighted translation semigroup with symbol $\varphi .$  It is easy to check that for every $t\in \mathbb R_+,$ the operator $S_t$ is injective. Further, if we impose a condition  that for every $t\in \mathbb R_+, \inf_x \frac{\varphi(x+t)}{\varphi(x)} > 0,$ then the range of $S_t$ is closed, implying that operator $S_t$ is left invertible for every $t\in \mathbb R_+.$  We say that the {\it semigroup $\{S_t\}$ is left invertible} if every operator $S_t$ in that semigroup is left invertible. Thus the stated condition on $\varphi $ ensures the left invertibility of the semigroup $\{S_t\}.$ 
In this case, the operator $(S_t^*S_t)^{-1}S_t^*$ is a left inverse of $S_t.$ 
Observe that for any $t\in \mathbb R_+,$  
$$\displaystyle (S_t^*S_t)^{-1}f(x)=\frac{\varphi(x)}{\varphi(x+t)}f(x),~\forall x\geq 0.$$
Now it is easy to see that the Cauchy dual  $S_t^{\prime}$ of $S_t$ is given by  
\begin{equation*}
S_t^{\prime}f(x) =
\begin{cases}
\displaystyle \frac{1}{\varphi_t(x)}f(x-t) & \text{if~ $x\geq t$},\\
0 & \text{if~ $x<t$}.
\end{cases}
\end{equation*}
 
Observe that for $t\in \mathbb R_+,$ the family of operators $\{S_t^{\prime}\}$ also forms a semigroup. 
We say that the weighted translation semigroup $\{S_t^{\prime}\}$ is a {\it Cauchy dual of the weighted translation semigroup $\{S_t\}.$ }
\begin{remark} \label{p13} 
Note that if $\{S_t\}$ is a weighted translation semigroup with symbol $\varphi,$ then the symbol corresponding to the weighted translation semigroup $\{S_t^{\prime}\}$ is $\frac{1}{\varphi}.$ 
Observe that if the weight function of a left invertible operator $S_t$ is $ \varphi_t,$ then the weight function of its Cauchy dual $S_t^\prime$ is $\frac{1}{\varphi_t}.$ This resembles with the fact that the weight sequence of the Cauchy dual of a left invertible weighted shift $T:\{\alpha_n\}$ is $\{\frac{1}{\alpha_n}\}.$ 
\end{remark} 

\begin{remark} \label{p36} If $\{S_t\}$ is a hyperexpansive weighted translation semigroup with symbol $\varphi,$ then $\varphi$ is an increasing function and $\inf_x \frac{\varphi(x+t)}{\varphi(x)}\geq 1,$ for all $t\in \mathbb R_+.$ Thus every hyperexpansive weighted translation semigroup is left invertible. \end{remark}

It is known that if $T$ is a completely hyperexpansive weighted shift, then its Cauchy dual $T^\prime$ is a subnormal contraction \cite[Remark 4, Proposition 6]{At}. 
We now prove this result in the context of a weighted translation semigroup.
\begin{proposition} \label{p14} If the semigroup $\{S_t\}$ with symbol $\varphi\in C^\infty$ is a completely hyperexpansive weighted translation semigroup, then its Cauchy dual $\{S_t^\prime \}$ is a subnormal contraction.  
\end{proposition}   
\begin{proof} In view of  Corollary \ref{p6} (3),(4) and Remark \ref{p13}, it is sufficient to prove that if $\varphi$ is a completely alternating function, then the function $\frac{1}{\varphi}$ is completely monotone. This result is a special case of \cite[Theorem 3.6(ii)]{SSV}. However, we give a direct and simple proof of this fact. 
Let $\psi =\frac{1}{\varphi}.$
As $\varphi$ is a completely alternating function, derivatives of $\varphi$ of odd orders are non-negative and derivatives of $\varphi$ of even orders are non-positive.
We need to show that derivatives of $\psi $ of odd orders are non-positive and derivatives of $\psi$ of even orders are non-negative.
Now $$ \psi ^\prime =\frac{-1}{\varphi ^2}\varphi ^\prime \leq 0,~~ \psi ^{\prime \prime} =\frac{-1}{\varphi ^2}\varphi ^{\prime \prime} + \frac{2}{\varphi ^3}{\varphi ^\prime}^2 \geq 0.$$
Suppose $k$ and $n$ are positive integers.
Suppose $\psi ^{(k)}\leq 0$ if $k$ is odd and $\psi ^{(k)}\geq 0$ if $k$ is even for $k<n.$
We now want to prove that $\psi ^{(n)}\leq 0$ if $n$ is odd and $\psi ^{(n)}\geq 0$ if $n$ is even.
Note that $\psi \varphi =1.$ Therefore $(\psi \varphi)^{(n)} =0.$
We have $$0=\psi ^{(n)}\varphi +{n\choose 1}\psi ^{(n-1)}\varphi ^\prime +{n\choose 2}\psi ^{(n-2)}\varphi ^{\prime \prime} +\cdots +{n\choose {n-1}}\psi ^\prime \varphi ^{(n-1)}+\psi \varphi ^{(n)}.$$
Therefore  
$$\psi ^{(n)}\varphi =-{n\choose 1}\psi ^{(n-1)}\varphi ^\prime -{n\choose 2}\psi ^{(n-2)}\varphi ^{\prime \prime} -\cdots -{n\choose {n-1}}\psi ^\prime \varphi ^{(n-1)}-\psi \varphi ^{(n)}. $$
Suppose $n$ is odd, then $n-1$ is even. Therefore $\psi ^{(n-1)}\geq 0,\varphi ^\prime \geq 0.$ Since $n-2$ is odd, $\psi ^{(n-2)}\leq 0,\varphi ^{\prime \prime} \leq 0.$ Continuing this argument, $\psi ^\prime \leq 0, \varphi ^{(n-1)}\leq 0$ and $\psi > 0,\varphi ^{(n)}\geq 0.$
Since each term on right hand side is non-positive, the left hand side $\psi ^{(n)}\varphi \leq 0,$ but $\varphi > 0.$ Hence $\psi ^{(n)} \leq 0.$
Suppose $n$ is even, then $n-1$ is odd. Therefore $\psi ^{(n-1)}\leq 0,\varphi ^\prime \geq 0.$ Since $n-2$ is even, $\psi ^{(n-2)} \geq 0,\varphi ^{\prime \prime} \leq 0.$ Continuing this argument, $\psi ^\prime \leq 0, \varphi ^{(n-1)}\geq 0$ and $\psi > 0,\varphi ^{(n)}\leq 0.$
Since each term on right hand side is non-negative, the left hand side $\psi ^{(n)}\varphi \geq 0,$ but $\varphi >0.$ Therefore $\psi ^{(n)} \geq 0.$ Hence $ \psi =\frac{1}{\varphi}$ is a completely monotone function.
\end{proof}

\begin{remark} \label{p15}
The technique used in the proof of \ref{p14} allows one to give a simpler proof of the fact that if $T$ is a completely hyperexpansive weighted shift, then its Cauchy dual $T^\prime$ is a subnormal contraction. 
Recall that a weighted shift operator $T:\{\alpha_n\},$ with weight sequence $\{\alpha_n\}$ on a Hilbert space $H,$ with orthonormal basis $\{e_n\}_{n=0}^\infty$ is defined by $Te_n=\alpha_ne_{n+1}.$ The sequence $\{\beta_n\}$ associated with $T$ is defined by $\beta_0=1,\beta_n=\prod_{k=0}^{n-1}\alpha_k^2~~(n\geq 1).$ Note that $\beta_n=||T^ne_0||^2$ and $\alpha_n=\sqrt{\frac{\beta_{n+1}}{\beta_n}},~n\geq 0.$ Recall that the Cauchy dual of a weighted shift operator $T:\{\alpha_n\},$ is a weighted shift operator $T^\prime :\{\frac{1}{\alpha_n}\}.$
Recall that a weighted shift $T$ is completely hyperexpansive (subnormal contraction) if and only if the sequence $n\rightarrow \beta_n$ is completely alternating (completely monotone). The forward difference operator on $\mathbb N$ is defined as $(\Delta \phi)(n)=\phi (n+1)-\phi (n).$ 
The map $\phi$ is completely alternating if and only if $(-1)^k\Delta ^k \phi (n)\leq 0$ and the map $\phi$ is completely monotone if and only if $(-1)^k\Delta ^k \phi (n)\geq 0.$ Let $\varphi(n)=\beta_n$ and $\psi(n)=\frac{1}{\varphi(n)}.$
A formula similar to Leibnitz theorem for the derivative of a product of two functions is as follows:
$$(\Delta ^n \phi \psi)(x)=\sum_{k=0}^n {n\choose k} (\Delta ^k \phi) (x)(\Delta ^{n-k} \psi )(x+k).$$
Now using the argument similar to that in the proof of Proposition \ref{p14}, it can be proved that if the sequence $\varphi(n)$ is completely alternating, then the sequence $\psi(n)=\frac{1}{\varphi(n)}$ is completely monotone. \end{remark}

\begin{remark}
In view of Corollary \ref{p6} and Proposition \ref{p14}, it is instructive to account for a procedure to generate a weighted shift operator from a weighted translation semigroup and vice versa in the context of a completely hyperexpansive weighted translation semigroup and its Cauchy dual.  
Let $\{S_t\}$ be a completely hyperexpansive weighted translation semigroup with symbol $\varphi \in C^\infty. $ Define  sequences $\beta_n$ and $\alpha_n$ by  $$\beta_n=\varphi(n)\  {\rm and}\  \alpha_n=\sqrt{\frac{\beta_{n+1}}{\beta_n}}=\sqrt{\frac{\varphi(n+1)}{\varphi(n)}}.$$ It is clear that the weighted shift operator with weight sequence $\{\alpha_n\}$ is completely hyperexpansive. 
For the reverse process, we need to start with a completely hyperexpansive weighted shift operator $T$ such that the corresponding completely alternating sequence $\{\beta_n\}_{n\geq 0}$ is {\it minimal .} Recall that a sequence $\{\beta_n\}$ is said to be minimal if the sequence $\{\beta_0,\beta_1-\epsilon,\beta_2-\epsilon,\cdots \}$ is not completely alternating for any positive $\epsilon.$  By an application of \cite[Theorem 1]{AR}, there exists a completely alternating function $\varphi $ on $\mathbb R_+$ satisfying $\varphi(n)=\beta_n.$ Now  the semigroup $\{S_t\}$ with symbol $\varphi$ is clearly a  completely hyperexpansive weighted translation semigroup. 
One may apply a similar procedure with complete hyperexpansion replaced by a subnormal contraction with appropriate changes in the definition of minimality of a completely monotone sequence. The details can be carried out by using \cite[Chapter 4, Definition 14a, Theorem 14b]{W-1}. 
The following diagram depicts the association of completely hyperexpansive weighted shift operators and completely hyperexpansive weighted translation semigroups as well as their respective Cauchy duals, as described above:
\begin{center}
\begin{tabular}{ccccccc}  
$T:\{\alpha_n\}$ & $\longleftrightarrow$ & $\{\beta_n\}$ & $\longleftrightarrow$ &  $\varphi(x)$ & $\longleftrightarrow$ & $\{S_t\}$ \\
 & & & & & & \\
$\big\updownarrow$ &                  &        &    &     &    &  $\big\updownarrow$  \\
 & & & & & & \\ 
$T^\prime:\{\frac{1}{\alpha_n}\}$ & $\longleftrightarrow$ & $\{\frac{1}{\beta_n}\}$ & $\longleftrightarrow$ &  $\frac{1}{\varphi(x)}$ & $\longleftrightarrow$ & $\{S_t^\prime \}$
\end{tabular}
\end{center}

\end{remark}
In the light of the positive result in case of a weighted shift operator, one might expect that the Cauchy dual of a completely hyperexpansive operator is a subnormal contraction. Though the problem still remains unsettled, in the special case of a 2-isometry, it was proved that the Cauchy dual of a 2-isometric weighted shift is a subnormal contraction \cite[Theorem 2.5(3)]{AC}. The proof of this fact in the special case of a 2-isometric weighted translation semigroup is trivial.     

We now turn our attention to the class 2-hyperexpansive weighted translation semigroups. Note that the class of 2-hyperexpansive operators is strictly bigger than the class of completely hyperexpansive operators. In \cite{Sh1} and \cite[Theorem 2.9]{Ch}, it is proved that if $T$ is a 2-hyperexpansive operator, then its Cauchy dual is a hyponormal contraction. We now present a special case of this result for a 2-hyperexpansive weighted translation semigroup. 
\begin{proposition} \label{p16} If the semigroup $\{S_t\}$ with symbol $\varphi \in C^2$ is a \\2-hyperexpansive weighted translation semigroup, then its Cauchy dual $\{S_t^\prime \}$ is a hyponormal contraction.\end{proposition}  
\begin{proof} In view of the Corollary \ref{p6} (1), Theorem \ref{p2} and Remark \ref{p13}, it is sufficient to prove that if $\varphi$ is a concave function, then $\log \frac{1}{\varphi}$ is a convex function. Since $\varphi$ is a concave function, $\varphi ^{\prime \prime} \leq 0.$
Let $\psi =\log \frac{1}{\varphi}.$
Then $$\psi ^\prime =\frac{-1}{\varphi }\varphi ^\prime ,~~ \psi ^{\prime  \prime} =\frac{1}{\varphi ^2}{\varphi ^\prime}^2-\frac{1}{\varphi }\varphi ^{\prime \prime} \geq 0.$$
Hence $\psi =\log \frac{1}{\varphi}$ is a convex function.
\end{proof}
 
\begin{remark} \label{p40} If $\{S_t\}$ is a weighted translation semigroup, then the semigroup property implies that for any non-negative integer $n,$ the operator $S_t^n= S_{nt}$ again belongs to the semigroup  $\{S_t\}.$ In particular, if $\{S_t\}$ is a hyponormal semigroup, the each operator $S_t$  is power hyponormal. Thus the Cauchy dual $S_t^\prime$ of a 2-hyperexpansive operator $S_t$ is power hyponormal. \end{remark}

Note that the Cauchy dual of a completely hyperexpansive weighted translation semigroup is a subnormal contraction (Proposition \ref{p14}) and the Cauchy dual of a 2-hyperexpansive weighted translation semigroup is a hyponormal contraction (Proposition \ref{p16}). Thus in this context, we now present an example of a 2-hyperexpansive weighted translation semigroup $\{S_t\}$ whose Cauchy dual is not a subnormal contraction. In view of the Remark \ref{p40}, the Cauchy dual operator $S_t^\prime$ in the following example is power hyponormal which is not subnormal.
\begin{example} Let $\{S_t\}$ be a weighted translation semigroup with symbol $\varphi$ given by  $\varphi(x)=2x-\log(\cosh(x-10))+100.$ Observe that $\varphi >0$ and $\varphi \in C^\infty.$ Note that $\varphi^\prime(x)=2-\tanh(x-10)\geq 0$ and 
$\varphi^{\prime\prime}(x)={\tanh}^2(x-10)-1\leq 0.$ Therefore the function $\varphi$ is concave implying that the semigroup $\{S_t\}$  is 2-hyperexpansive. Now 
$$\varphi^{\prime\prime\prime}(x)=2\tanh(x-10)(1-{\tanh}^2(x-10)).$$ Observe that $\varphi^{\prime\prime\prime}(11)>0$ and $\varphi^{\prime\prime\prime}(0)<0.$ Hence the function $\varphi$ is not completely alternating. Consequently, the semigroup $\{S_t\}$ is not completely hyperexpansive.
 Let $\psi(x)=\frac{1}{\varphi(x)}.$ Then $$\psi^\prime(x)=-\frac{1}{\varphi^2(x)}\varphi^\prime(x) \leq 0,~\psi^{\prime\prime}(x)=-\frac{1}{\varphi^2(x)}\varphi^{\prime\prime}(x)+\frac{2}{\varphi^3(x)}\varphi^\prime \geq 0$$ and 
$$\psi^{\prime\prime\prime}(x)=\frac{1}{\varphi^2(x)}\left(-\varphi^{\prime\prime\prime}(x)+\frac{4}{\varphi(x)}\varphi^{\prime\prime}(x)-\frac{6}{\varphi^2(x)}\varphi^\prime(x)\right).$$
By direct computations we get that $\psi^{\prime\prime\prime}(11)<0$ and $\psi^{\prime\prime\prime}(0)>0.$ Hence the function $\psi$ is not completely monotone. Thus the semigroup $\{S_t^\prime\}$ is not a subnormal contraction. \eop 
\end{example}
The problem of describing the Cauchy dual of an alternatingly hyperexpansive operator, even in the special case of a weighted translation semigroup seems to be difficult.  The following examples illustrate the situation in some simple cases. In view of  Corollary \ref{p6}(5) and Remark \ref{p13}, we need to examine the reciprocal of an absolutely monotone function.   
\begin{example} \
\begin{enumerate}
\item Let $\varphi(x)=e^x.$ It is an absolutely monotone function. Observe that the function $\frac{1}{\varphi(x)}=e^{-x}$ is completely monotone.
\item
The following result is a special case of \cite[Proposition 4.3]{AC}. Let $\varphi(x)=x^2+px+q,~p,q>0.$ It is an absolutely monotone function. Suppose $\varphi(x)$ has both complex roots. 
It is easy to see that the function $\frac{1}{\varphi}$ is not completely monotone. 
\item Let $\varphi(x)=x^2+px+q,~p,q>0.$ Suppose $\varphi(x)$ has both real roots. It can be seen that the function $\frac{1}{\varphi}$ is completely monotone. 
\eop
\end{enumerate}
\end{example}

\section{Concluding Remarks} 
The present work attempts to study the behaviour of an operator $S_t$ with a motivation to compare it with a weighted shift operator. The appearance of a function in place of a sequence allows one to use different techniques, not available in the discrete case. The authors believe that these techniques might be employed to tackle some problems about weighted shifts. In particular, the problem of characterizing $m$-isomeric transformations whose Cauchy dual is a contractive subnormal operator is challenging even in the case of a weighted shift operator. Though some special cases of this problem have been discussed in \cite{AC}, the general case still remains unanswered. 
In a separate article as a sequel to this, authors have developed an analytic model for left invertible weighted translation semigroups enabling one to realize every left invertible operator $S_t$  as a multiplication by $z$ on some suitable reproducing kernel Hilbert space. In particular the model applies to hyperexpansive weighted translation semigroups. Further we shall also describe the spectral picture of a left invertible weighted translation semigroup. In the process, we shall point out that in contrast with the situation for weighted shift operators, the kernel of the adjoint $S_t^*$ is infinite dimensional. 

\subsection*{Acknowledgment} The authors are thankful to an unknown referee for pointing out some corrections in the original version and offering some useful suggestions. The authors would also like to thank Sameer Chavan for several useful discussions throughout the preparation of this paper.

\end{document}